\documentclass[reqno,10pt]{amsart}
\usepackage[dvipsnames,usenames]{color} 
\usepackage[toc,page]{appendix}

\usepackage{tikz} 
\usepackage{mathcomp,wasysym}  

\usepackage{hyperref}
\usepackage[mathscr]{euscript}  
        
\usepackage{cite}      
\usepackage{graphicx}  
\usepackage[all]{xy} \xyoption{arc} \xyoption{color}

\usepackage{amsmath} 
\usepackage{amsthm} 
\usepackage{amsfonts}  
\usepackage{amssymb} 

\usepackage{verbatim}

\DeclareMathAlphabet{\mathpzc}{OT1}{pzc}{m}{it}

\newcommand{\vertiii}[1]{{\left\vert\kern-0.25ex\left\vert\kern-0.25ex\left\vert #1 
    \right\vert\kern-0.25ex\right\vert\kern-0.25ex\right\vert}}

\numberwithin{equation}{section} 

\newtheorem{theorem}{Theorem}[section]

\def\bbR{{\mathbb R}}

\newcommand{\pare}[1]{\left( #1 \right)}

\def\p{{\partial\hspace{1pt}}}

\def\comm#1#2{{\big[\hspace{-3.2pt}\big[#1,#2\hspace{1pt}\big]\hspace{-3.2pt}\big]}}

\def\jump#1{{[\hspace{-2pt}[#1]\hspace{-2pt}]}}

\title{Interfaces in incompressible flows}

\author[R.Granero-Belinch\'{o}n]{Rafael Granero-Belinch\'{o}n}
\email{rafael.granero@unican.es}
\address{Departamento  de  Matem\'aticas,  Estad\'istica  y  Computaci\'on, Universidad  de Cantabria, Santander,  Spain}

\begin{document}

\begin{abstract}
The motion of both internal and surface waves in incompressible fluids under capillary and gravity forces is a major research topic. In particular, we review the derivation of some new models describing the dynamics of gravity-capillary nonlinear waves in incompressible flows. These models take the form of both bidirectional and unidirectional nonlinear and nonlocal wave equations. More precisely, with the goal of telling a more complete story, in this paper we present the results in the works \cite{GS,aurther2019rigorous,granero2019models,granero2019well,granero2020well,granero2021motion,
granero2021global} together with some new results regarding the well-posedness of the resulting PDEs.
\end{abstract}

\maketitle
{\small
\tableofcontents}
\section{Introduction}
Water waves have been the object of study of many different researchers along History. This is a \emph{familiar} problem in the sense that everyone has seen water waves in a beach at some point. However, this is also a very challenging problem. Actually, in the Feynmann Lecture on Physics there is already a warning about the difficulty of the motion of water waves:
\begin{quote}
``\emph{Now, the next waves of interest, that are easily seen by everyone and which are usually used as an example of waves in elementary courses, are water waves. As we shall soon see, they are the worst possible example, because they are in no respects like sound and light; they have all the complications that waves can have.}".\\
\emph{The Feynman Lectures on Physics Vol. I Ch. 51: Waves}, Richard Feynmann 
\end{quote}

In \emph{most} of the cases, it is assumed that water can be approximated as an incompressible, irrotational and inviscid fluid. In two dimensions, the evolution of such a fluid is given by the Euler equations
\begin{align*}
\rho\left(u_t+(u\cdot\nabla)u\right)+\nabla p+g\rho e_2&=0\quad x\in \Omega,\,t\in[0,T]\\
\partial_t \rho+ \nabla \cdot(u\rho)&=0\quad x\in \Omega,\,t\in[0,T]\\
\nabla\cdot u&=0\quad x\in \Omega,\,t\in[0,T]\\
\nabla\times u&=0\quad x\in \Omega,\,t\in[0,T],
\end{align*}
where, $\Omega\subset \bbR^2$ is the spatial domain, $T>0$ and $u=(u_1,u_2)$, $\rho$, $p$ and $g$ denote the incompressible velocity field, the density of the fluid, the pressure and the acceleration due to gravity, respectively. 

However, usually in real applications, there is one or several fluids filling several \emph{moving} domains separated by \emph{moving} interfaces. One can take the sea as an example. There, one fluid (water) moves in a (time-dependent) region that is bounded below by the sea's bottom and above by another fluid (air). The water wave is then the (time-dependent) \emph{interface} between these two fluids (water and air). In this type of problems (known in the literature as \emph{free boundary} problems), the (time-dependent) domain of equations, $\Omega=\Omega(t)$, is also an unknown of the problem, and, as such, it has to be recovered from the dynamics (see figure \ref{fig0}).

Due to its importance, the study of water waves is a classical research topic in Mathematics and also in Physics or Engineering \cite{craik2004origins,GraneroGaceta}. Actually, the mathematical study of water waves leads to new theories and tools like new functional inequalities (as Strichartz and Morawetz inequalities), new paralinearization techniques or a deeper study of the Dirichlet-to-Neumann operator.

Even if great mathematicians like Euler, Lagrange, Laplace or Cauchy contributed to our knowledge of the motion of water waves, we could argue that the modern mathematical study of water waves originates with the classical works of Stokes \cite{stokes1846report,stokes1880theory,craik2005george}. This research program was later continued by Boussinesq. In particular, Boussinesq's goal was to describe most of the fluid dynamics in certain physical regime while, at the same time, the problem becomes somehow simplified by neglecting unimportant effects. As he was mainly interested in shallow water waves, Boussinesq neglected terms whose contribution is not relevant for this case. By doing so he obtained new asymptotic models (such as the KdV equation for instance) that describe the main dynamics in the shallow water regime. Furthermore, Boussinesq also opened the door to find other partial differential equations that capture the fluid behaviour under many realistic possible hypotheses in the range of the physical parameters (depth of the fluid layer, the amplitude and wavelength of the wave, viscosity, etc). 

The purpose of this paper is to review some partial differential equations arising in the study of interfaces in incompressible flows and to present some new mathematical results. Although we will focus in models for the case of Euler and Navier-Stokes equations, similar models have been obtained and studied for waves in porous media in \cite{granero2019asymptotic,granero2020asymptotic}. We start with the models for waves with small steepness and viscosity in \cite{granero2019models,granero2019well,granero2020well,granero2021motion,granero2021global}. In particular, in section \ref{GSO} we present the models for viscous fluids with shear viscosiy in \cite{granero2019models,granero2019well,granero2020well,granero2021global} and the models for viscous fluids with odd or Hall viscosity in \cite{granero2021motion}. In section \ref{CGSW} we review the models in \cite{aurther2019rigorous} for inviscid fluids. In addition, we will review the models in \cite{GS} for internal waves in incompressible fluids in section \ref{GS}. Furthermore, we will present a new improvement of these models for the case of the Kelvin-Helmholtz instability. In section \ref{Num}, we will also present some numerical simulations comparing the models in \cite{GS} and \cite{aurther2019rigorous} with simulations of the full water wave system. In this section we also present new simulations of falling drops and rising bubbles. Later, in section \ref{new}, we present some new mathematical results regarding the well-posedness of some of the asymptotic models described in this paper. More precisely, we prove that, for small initial data in $H^1$, the solution of the unidirectional wave propagation for viscous fluids exists globally (see Theorem \ref{theorem1}). Finally, we also prove the local well-posedness for the unidirectional wave propagation for inviscid fluids and analytic initial data (see Theorem \ref{theorem2}). 

\subsection*{Notation}
In this paper we use the following notation for time derivatives and the commutator between an operator and the multiplication by a function
$$
f_t=\frac{\partial f}{\partial t },\quad 
\comm{\mathcal{T}}{f}g=\mathcal{T}(fg)-f\mathcal{T}(g).  
$$
We write
$$
\partial_1 f=\frac{\partial f}{\partial x_1}, \quad \partial_2 f=\frac{\partial f}{\partial x_2},\quad \nabla=\left(\partial_{1},\partial_2\right).
$$

Similarly, we denote $\varepsilon, \beta \text{ and } \alpha$ three dimensionless quantities that measure the nonlinearity, the effect of the surface tension and the effect of viscosity, respectively.

We also introduce the Hilbert transform
\begin{equation}\label{Hilbert}
H f( \alpha) = \frac{1}{\pi} P.V. \int_\bbR \frac{f (\beta)}{ \alpha - \beta } d \beta  \,.
\end{equation} 
This singular integral operator is the following multiplier operator in Fourier variables
$$
\widehat{\mathcal{H}f}(k)=-i\text{sgn}(k) \hat{f}(k).
$$
Finally, we introduce $\Lambda$, the square root of the Laplacian operator,  
$$
\widehat{\Lambda f}(k)=|k|\hat{f}(k).
$$

\section{Models of surface waves in viscous fluids: truncation in the steepness}\label{GSO}
For most of the applications in coastal engineering, water waves are assumed to be inviscid. This assumption, although physically absurd, lead to very accurate results. However, as already stated by the celebrated mathematician and oceanographer Longuet-Higgins \cite{longuet1992theory} 

\begin{quote}
\emph{For certain applications, however, viscous damping of the waves is important, and it would be highly convenient to have equations and boundary conditions of comparable simplicity as for undamped waves.}
\end{quote}

The purpose of this section is to review recent advances in this program for both the case of newtonian and non-newtonian fluids.

The equations describing the motion of a homogeneous incompressible fluid take the following form
\begin{align*}
u_t+(u\cdot\nabla)u-\nabla\cdot \mathscr{T}&=0 \quad x\in\Omega(t),\; t\in[0,T],\\
\nabla\cdot u&=0 \quad x\in\Omega(t),\; t\in[0,T],
\end{align*}
where $u$ and $\mathscr{T}$ denote the velocity and stress tensor of the fluid respectively. This stress tensor takes different forms depending on the physical properties of the fluid. The presence of viscosity creates a boundary layer where the vorticity is large. Outside this boundary layer beneath the surface wave, the flow is nearly irrotational \cite{lamb1932hydrodynamics} (see also \cite{boussinesq1895lois}). The idea is then to capture the effect of the boundary layer as a new viscous contribution at the free boundary. As a consequence, irrotational models with an adapted boundary condition at the free surface lead to accurate descriptions of the physical phenomena \cite{dias2008theory,jiang1996moderate,longuet1992theory,wu2006note}. 

\begin{figure}[h]
\begin{center} 
\begin{tikzpicture}[domain=0:3*pi, scale=0.7] 
\draw (pi,1.7) node { Air};
\draw[line width=4mm, smooth, color=blue] plot (\x,{0.2*sin(\x r)+1});
\fill[blue!10] plot[domain=0:3*pi] (\x,0) -- plot[domain=3*pi:0] (\x,{0.2*sin(\x r)+1});
\draw[very thick,<->] (3*pi+0.4,0) node[right] {$x$} -- (0,0) -- (0,2) node[above] {$z$};
\node[right] at (2*pi,1.7) {$\Gamma(t)$};
\node[right] at (3*pi+0.1,1) {Boundary layer};
\node[right] at (1,0.5) {Fluid};
\node[right] at (3,0.5) {$\Omega(t)$};
\end{tikzpicture}  
\end{center}
\caption{\emph{A representation of a surface wave in a fluid with viscosity. The blue curve is an illustration of the interface $\Gamma(t)$ and the boundary layer beneath the fluid.}}
 \label{fig1} 
\end{figure}
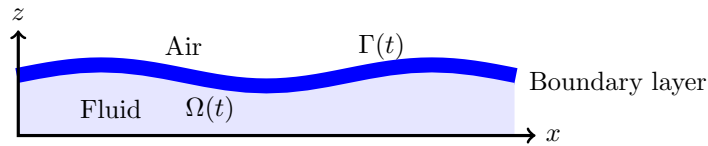

\subsection{The Zakharov formulation of the Euler system}
Let us first consider the motion of an irrotational perfect fluid bounded above by a surface wave. Then, the corresponding Euler system reads as follows
\begin{align}
\label{euler1}
\rho\left(u_t+(u\cdot\nabla)u\right)+\nabla p+g\rho e_2&=0\quad x\in \Omega(t),\,t\in[0,T]\\
\partial_t \rho+ u\cdot\nabla \rho&=0\quad x\in \Omega(t),\,t\in[0,T]\\
\nabla\cdot u&=0\quad x\in \Omega(t),\,t\in[0,T]\\
\nabla\times u&=0\quad x\in \Omega(t),\,t\in[0,T].
\end{align}
We assume that the surface wave can be described as the graph of certain function $h$. Then, the appropriate boundary conditions at the surface wave are
\begin{equation}\label{jump_cond1}
h_t=u\cdot(-\partial_1 h,1)\,, \ \ \ p=-\gamma\mathcal{K} \,, \ \ \text{ on } \Gamma(t) \,,
\end{equation}
where $\gamma>0$ is the surface tension coefficient and $\mathcal{K}$ is the curvature of the wave. As the fluid is irrotational, the velocity is described in terms of the velocity potential
$$
u=\nabla\phi,
$$
and, following the pioneer work of Zakharov \cite{zakharov1968stability}, we have that the free boundary irrotational Euler equations with surface tension are equivalent to the following system of PDEs in dimensionless form
\begin{subequations}\label{eq:zakharov}
\begin{align}
\Delta \phi&=0&&\text{ in }\Omega(t)\times[0,T],\\
\phi  &= \chi \qquad &&\text{ on }\Gamma(t)\times[0,T],\\
\chi_t &=-\frac{\varepsilon}{2}|\nabla \phi|^2- h+\frac{\beta\partial_{1}^2h}{\left(1+\left(\varepsilon\partial_{1}h\right)^2\right)^{3/2}} +\varepsilon\partial_2\phi\left(\nabla\phi\cdot(-\varepsilon\partial_1 h,1)\right)&&\text{ on }\Gamma(t)\times[0,T],\\
h_t&=\nabla\phi\cdot (-\varepsilon\partial_1 h,1)&&\text{ on }\Gamma(t)\times[0,T],
\end{align}
\end{subequations}
Then, in terms of the Zakharov formulation, the previous idea to capture the contribution of the viscous boundary layer below the surface wave is to modify equations (\ref{eq:zakharov}c) and (\ref{eq:zakharov}d) to take into consideration the viscous effects and the form of the stress tensor $\mathscr{T}$. By doing this, the resulting free boundary problem maintains the irrotationality condition while capturing viscous effects near the surface wave.

\subsection{Models for the case of Newtonian fluids}
The case of viscous newtonian fluids where the tensor takes the classical form
$$
\mathscr{T}^i_j=-p\delta^i_j+\nu_e\left(\nabla_ju_i+\nabla_i u_j\right).
$$
This viscosity is called \emph{shear} viscosity. In this case the equations take the following form
\begin{align*}
u_t+(u\cdot\nabla)u+\nabla p -\nu_e \Delta u&=0\quad x\in\Omega(t),\; t\in[0,T],,\\
\nabla\cdot u&=0\quad x\in\Omega(t),\; t\in[0,T].
\end{align*}
Due to the symmetries of the tensor and to distinguish if from other (non-newtonian) viscosity stresses, we will also call this stress \emph{even} viscosity. Of course, it is not surprising that dealing with the case of the standard free surface Navier-Stokes, there are plenty of works proposing different approaches to this problem (see for instance \cite{boussinesq1895lois,dias2008theory,lamb1932hydrodynamics,jiang1996moderate,longuet1992theory,
ruvinsky1991numerical,
wu2006note} among many others). On the one hand, we consider the model
\begin{subequations}\label{eq:1}
\begin{align}
\Delta \phi&=0&&\text{ in }\Omega(t)\times[0,T],\\
\phi  &= \chi \qquad &&\text{ on }\Gamma(t)\times[0,T],\\
\chi_t &=-\frac{\varepsilon}{2}|\nabla \phi|^2- h+\frac{\beta\partial_{1}^2h}{\left(1+\left(\varepsilon\partial_{1}h\right)^2\right)^{3/2}}-\alpha\partial_{2}^2\phi &&\nonumber\\
&\quad+\varepsilon\partial_2\phi\left(\nabla\phi\cdot(-\varepsilon\partial_1 h,1)\right)&&\text{ on }\Gamma(t)\times[0,T],\\
h_t&=\nabla\phi\cdot (-\varepsilon\partial_1 h,1)&&\text{ on }\Gamma(t)\times[0,T].
\end{align}
\end{subequations}
This and other similar models have been proposed and studied by Longuet-Higgins \cite{longuet1992theory}, Ruvinsky, Feldstein \& Freidman \cite{ruvinsky1991numerical} and Wu, Liu \& Yue \cite{wu2006note} (see also \cite{jiang1996moderate}). On the other hand, Dias, Dyachenko and Zakharov \cite{dias2008theory} proposed the following modification having dissipative effects also on the surface wave
\begin{subequations}\label{eq:2}
\begin{align}
\Delta \phi&=0&&\text{ in }\Omega(t)\times[0,T],\\
\phi  &= \chi \qquad &&\text{ on }\Gamma(t)\times[0,T],\\
\chi_t &=-\frac{\varepsilon}{2}|\nabla \phi|^2- h+\frac{\beta\partial_{1}^2h}{\left(1+\left(\varepsilon\partial_{1}h\right)^2\right)^{3/2}}-\alpha\partial_{2}^2\phi &&\nonumber\\
&\quad+\varepsilon\partial_2\phi\left(\nabla\phi\cdot(-\varepsilon\partial_1 h,1)+\alpha \partial_1^2 h\right)&&\text{ on }\Gamma(t)\times[0,T],\\
h_t&=\nabla\phi\cdot (-\varepsilon\partial_1 h,1)+\alpha\partial_{1}^2 h&&\text{ on }\Gamma(t)\times[0,T].
\end{align}
\end{subequations}
To have a system of PDEs that includes both approaches, we define the following free boundary problem \cite{granero2019models}:
\begin{subequations}\label{eq:all}
\begin{align}
\Delta \phi&=0&&\text{ in }\Omega(t)\times[0,T],\\
\phi  &= \chi \qquad &&\text{ on }\Gamma(t)\times[0,T],\\
\chi_t &=-\frac{\varepsilon}{2}|\nabla \phi|^2- h+\frac{\beta\partial_{1}^2h}{\left(1+\left(\varepsilon\partial_{1}h\right)^2\right)^{3/2}}-\alpha_1\partial_{2}^2\phi &&\nonumber\\
&\quad+\varepsilon\partial_2\phi\left(\nabla\phi\cdot(-\varepsilon\partial_1 h,1)+\alpha_2 \partial_1^2 h\right)&&\text{ on }\Gamma(t)\times[0,T],\\
h_t&=\nabla\phi\cdot (-\varepsilon\partial_1 h,1)+\alpha_2\partial_{1}^2 h&&\text{ on }\Gamma(t)\times[0,T],
\end{align}
\end{subequations}
Our goal now is to derive new asymptotic models for the previous system \eqref{eq:all} modelling viscous water waves. In order to do that we introduce the following diffeomorphism
$$
\Psi(x_1,x_2,t)=(x_1,x_2+\varepsilon h(x_1,t)).
$$
We observe that this diffeomorphism maps the reference domain $\Omega=\mathbb{R}^2_-$ onto the moving domain $\Omega(t)$,
$$
\Psi:\Omega\mapsto\Omega(t).
$$
Using this diffeomorphism to fix the domain together with the chain rule (see for instance \cite{aurther2019rigorous, Coutand-Shkoller:well-posedness-free-surface-incompressible, granero2019asymptotic, Lannes1, ngom2018well}), we find the following Arbitrary Lagrangian-Eulerian formulation of \eqref{eq:all}
\begin{subequations}\label{eq:ALE}
\begin{align}
A^\ell_j\partial_\ell\left(A^k_j\partial_k \Phi\right)&=0&&\text{ in }\Omega\times[0,T],\\
\Phi  &= \chi \qquad &&\text{on }\Gamma\times[0,T],\\
\chi_t &=-\frac{\varepsilon}{2}A^k_j\partial_k\Phi A^\ell_j\partial_\ell\Phi- h+\frac{\beta\partial_{1}^2h}{\left(1+\left(\varepsilon\partial_{1}h\right)^2\right)^{3/2}} &&\nonumber\\
&\quad+\varepsilon A^k_2\partial_k\Phi\left(A^\ell_j\partial_\ell\Phi A^2_j+\alpha_2 \partial_1^2 h\right)-\alpha_1A^\ell_2\partial_\ell\left(A^k_2\partial_k \Phi\right)&&\text{ on }\Gamma\times[0,T],\\
h_t&=A^k_j\partial_k\Phi A^2_j+\alpha_2\partial_{1}^2 h&&\text{ on }\Gamma\times[0,T],
\end{align}
\end{subequations}
where
\begin{align*}
\Phi = \phi \circ \Psi\;\text{ and }A=(\nabla\Psi)^{-1},
\end{align*}
and the reference boundary is $\Gamma=\mathbb{R}$. Using the explicit value of the cofactor matrix $A$, we write the previous system \eqref{eq:all} in the following form
\begin{subequations}\label{eq:all2}
\begin{align}
\Delta \Phi &={\varepsilon\left(\partial_1^2 h \ \partial_2 \Phi + 2\partial_1 h \ \partial_{12}\Phi\right)-\varepsilon^2(\partial_1 h)^2\partial^2_2\Phi}  ,  &&\text{in } \Omega\times[0,T]\,,\\
\Phi  &= \chi \qquad &&\text{on }\Gamma\times[0,T],\\
\chi_t &=-\frac{\varepsilon}{2}\left[(\partial_1\Phi)^2+(\varepsilon \partial_1 h\partial_2\Phi)^2+(\partial_2\Phi)^2-2\varepsilon \partial_1 h\partial_2\Phi\partial_1\Phi\right]&&\nonumber\\
&\quad- h+\frac{\beta\partial_{1}^2h}{\left(1+\left(\varepsilon\partial_{1}h\right)^2\right)^{3/2}}-\alpha_1\partial_2^2 \Phi &&\nonumber\\
&\quad+\varepsilon \partial_2\Phi\left(-\varepsilon \partial_1 h \partial_1 \Phi+\varepsilon^2 (\partial_1 h)^2 \partial_2\Phi +\partial_2\Phi+\alpha_2 \partial_1^2 h\right)&&\text{ on }\Gamma\times[0,T],\\
h_t&=-\varepsilon \partial_1 h \partial_1 \Phi+\varepsilon^2 (\partial_1 h)^2 \partial_2\Phi +\partial_2\Phi+\alpha_2\partial_{1}^2 h&&\text{ on }\Gamma\times[0,T].
\end{align}
\end{subequations}
We observe that the different order of $\varepsilon$ appear explicitly in this latter form, being this the main benefit of this formulation. Motivated by the already mentioned program initiated by Longuet-Higgins \cite{longuet1992theory}, we are interested in a model approximating \eqref{eq:all2} with an error $\mathcal{O}(\varepsilon^2)$. To find such a model, we introduce the following ansatz:
\begin{equation}
\label{eq:ansatz}
\begin{aligned}
\Phi(x_1,x_2,t) & = \sum_n \varepsilon^n \Phi^{(n)}(x_1,x_2,t), \\
\chi(x_1,t) & = \sum_n \varepsilon^n \chi^{(n)}(x_1,t), \\
h(x_1,t) & = \sum_n \varepsilon^n h^{(n)}(x_1,t).
\end{aligned}
\end{equation} 
With this ansatz we can split the nonlinear system \eqref{eq:all2} in an equivalent sequence of linear problems where the evolution of the $ k $-th term in the series is determined by the evolution of the preceding $ k-1 $ unknowns. 

Using that
$$
\frac{1}{\left(1+y^2\right)^{3/2}}=1+\mathcal{O}(y^2),
$$
we find that the first two unknowns solve
\begin{subequations}\label{eq:n0s2}
\begin{align}
\Delta \Phi^{(0)} &=0 ,  &&\text{in } \Omega\times[0,T]\,,\\
\Phi^{(0)}  &= \chi^{(0)}  \qquad &&\text{on }\Gamma\times[0,T],\\
\chi^{(0)} _t &=- h^{\pare{0}} +\beta\partial_{1}^2h^{\pare{0}} -\alpha_1^{2}\partial_2^2 \Phi^{(0)}  &&\text{ on }\Gamma\times[0,T],\\
h_t^{(0)} &=\partial_2\Phi^{(0)} +\alpha_2\partial_{1}^2 h^{\pare{0}} &&\text{ on }\Gamma\times[0,T].
\end{align}
\end{subequations}
and
\begin{subequations}\label{eq:n1s2}
\begin{align}
\Delta \Phi^{(1)} &=\partial_1^2 h^{\pare{0}} \ \partial_2 \Phi^{(0)} + 2\partial_1 h^{\pare{0}} \ \partial_{12}\Phi^{(0)}  ,  &&\text{in } \Omega\times[0,T]\,,\\
\Phi^{(1)}  &= \chi^{(1)} \qquad &&\text{on }\Gamma\times[0,T],\\
\chi^{(1)}_t &=\frac{1}{2}\left[(\partial_2\Phi^{(0)})^2-(\partial_1\Phi^{(0)})^2\right]&&\nonumber\\
&\quad- h^{\pare{1}}+\beta\partial_{1}^2h^{\pare{1}}-\alpha_1^{2}\partial_2^2\Phi^{(1)} +\alpha_2\partial_2\Phi^{(0)} \partial_1^2 h^{\pare{0}}&&\text{ on }\Gamma\times[0,T],\\
h^{\pare{1}}_t&=-\partial_1 h^{\pare{0}} \partial_1 \Phi^{(0)} +\partial_2\Phi^{(1)}+\alpha_2\partial_{1}^2 h^{\pare{1}}&&\text{ on }\Gamma\times[0,T].
\end{align}
\end{subequations}
These problems are linear and then they can be explicitly solved. After certain (long) computations \cite{granero2019models}, we find that the truncation
$$
h=h^{(0)}+\varepsilon h^{(1)}
$$
solves the so-called $h-$model for viscous water waves
\begin{multline}\label{models2}
h_{tt}-(\alpha_1+\alpha_2)\partial_{1}^2 h_t+ \Lambda h+\beta\Lambda^3 h+\alpha_1\alpha_2\partial_{1}^4 h\\
= \varepsilon\bigg\lbrace-\Lambda\left(\left(H h_t\right)^2\right)+\partial_1\comm{H}{h}\Lambda h +\beta\partial_1\comm{H}{h}\Lambda^3 h
\\
+\alpha_2\partial_1\comm{H}{H h _t}H \partial_1 ^2 h+\alpha_2\Lambda\left(H h _t H \partial_1 ^2 h \right)
+\alpha_1\alpha_2\partial_1\comm{\partial_1^2}{h}\Lambda\partial_{1} h\\
-\alpha_1\partial_1\comm{\partial_1^2}{h} H h _t
-\alpha_2^2\partial_1\comm{H}{\partial_{1}^2 h}\partial_{1}^2 h \bigg\rbrace,
\end{multline}
up to an $\mathcal{O}(\varepsilon^2)$ correction term. We observe that, when $\alpha_2=\alpha_1$, equation \eqref{models2} is an asymptotic model of the damped water waves system proposed by Dias, Dyachenko, and Zakharov \cite{dias2008theory}. Similarly, if $\alpha_2=0$, equation \eqref{models2} is an asymptotic model of the viscous water waves system proposed by Wu, Liu \& Yue \cite{wu2006note}.

Equation \eqref{models2} is a bidirectional wave equation of nonlocal type. We can also obtain a unidirectional model using a standard far-field change of variables
$$
\xi=\alpha-t,\quad \tau=\varepsilon t.
$$
Then, following the computation in \cite{granero2021global} and changing back the notation for our space and time variables, we find that, up to an $O(\varepsilon^2)$ correction, $f=\Lambda h$ in the far field variables solves
\begin{multline}\label{eq:univiscous}
2\varepsilon f_{t}=\mathcal{N}\partial_1 f+(\alpha_1+\alpha_2)\mathcal{N}\partial_1^2f+ \mathcal{N}H f-\beta \mathcal{N}H \partial_1^2 f+\alpha_1\alpha_2\mathcal{N}\partial_1^3 f\\
-\varepsilon \mathcal{N}\bigg{\lbrace}2f\partial_1 f+\Lambda\comm{H}{\Lambda^{-1}f}f +\beta\Lambda\comm{H}{\Lambda^{-1}f}\Lambda^2 f
\\
-\alpha_2\Lambda\comm{H}{f}\partial_1 f+\alpha_2\partial_1\left(f\partial_1 f \right)
+\alpha_1\alpha_2\Lambda\comm{\partial_1^2}{\Lambda^{-1}f}\partial_1f\\
+\alpha_1\Lambda\comm{\partial_1^2}{\Lambda^{-1}f} f
-\alpha_2^2\Lambda\comm{H}{\Lambda f}\Lambda f\bigg{\rbrace},
\end{multline}
where the operator $\mathcal{N}$ is defined as follows
$$
\mathcal{N}=\left(1-\left(\frac{\alpha_1+\alpha_2}{2}\right)^2\partial_1^2\right)^{-1}\left(1-\left(\frac{\alpha_1+\alpha_2}{2}\right)\partial_1\right),
$$
We observe that the terms $O(\varepsilon \alpha_1\alpha_2)$ and $O(\varepsilon \alpha_2^2)$ are much smaller than the other contributions. Neglecting now the nonlinear terms that are $O(\varepsilon \alpha_1\alpha_2)$ and $O(\varepsilon \alpha_2^2)$ we conclude
\begin{multline}\label{eq:univiscous2}
2\varepsilon f_{t}=\mathcal{N}\partial_1 f+(\alpha_1+\alpha_2)\mathcal{N}\partial_1^2f+ \mathcal{N}H f-\beta \mathcal{N}H \partial_1^2 f+\alpha_1\alpha_2\mathcal{N}\partial_1^3 f\\
-\varepsilon \mathcal{N}\bigg{\lbrace}2f\partial_1 f+\Lambda\comm{H}{\Lambda^{-1}f}f +\beta\Lambda\comm{H}{\Lambda^{-1}f}\Lambda^2 f
\\
-\alpha_2\Lambda\comm{H}{f}\partial_1 f+\alpha_2\partial_1\left(f\partial_1 f \right)
+\alpha_1\Lambda\comm{\partial_1^2}{\Lambda^{-1}f} f
\bigg{\rbrace},
\end{multline}

\subsection{Models for the case of Non-Newtonian fluids}
Non-newtonian fluids in which both time reversal and parity are broken can display a dissipationless viscosity that is odd under each of these symmetries \cite{abanov2018odd,abanov2020hydrodynamics,abanov2019free,avron1998odd,banerjee2017odd,ganeshan2017odd,ganeshannon,
soni2019odd,souslov2019topological,lapa2014swimming}. This viscosity is called \emph{odd} or Hall viscosity.  In this case the stress tensor takes the following form
$$
\mathscr{T}^i_j=-p\delta^i_j+\nu_o\left(\nabla_i u^\perp_j+\nabla^\perp_i u_j\right),
$$
where $a^\perp=(a_2,-a_1).$ The corresponding equations take the following form
\begin{align*}
u_t+(u\cdot\nabla)u+\nabla p-\nu_o\Delta u^\perp&=0 \quad x\in\Omega(t),\; t\in[0,T],\\
\nabla\cdot u&=0\quad x\in\Omega(t),\; t\in[0,T].
\end{align*}
As before, after considering the appropriate boundary term to capture the viscous effect at the boundary layer, we are left with the following free boundary problem in dimensionless formulation \cite{abanov2018odd,abanov2019free,ganeshannon,granero2021motion}
\begin{subequations}\label{eq:alldimensionless}
\begin{align}
\Delta \phi&=0&&\text{ in }\Omega(t)\times[0,T],\\
\phi  &= \chi \qquad &&\text{ on }\Gamma(t)\times[0,T],\\
\chi_t &=-\frac{\varepsilon}{2}|\nabla \phi|^2- h+\frac{\beta\partial_{1}^2h}{\left(1+\left(\varepsilon\partial_{1}h\right)^2\right)^{3/2}} &&\nonumber\\
&\quad+\varepsilon\partial_2\phi\left(\nabla\phi\cdot(-\varepsilon\partial_1 h,1)\right)\nonumber\\
&\quad+\frac{\alpha}{(1+\varepsilon^2 \partial_1 h^2)^{1/2}}\partial_1\left(\frac{h_t}{(1+\varepsilon^2\partial_1 h^2)^{1/2}}\right)&&\text{ on }\Gamma(t)\times[0,T],\\
h_t&=\nabla\phi\cdot (-\varepsilon\partial_1 h,1) &&\text{ on }\Gamma(t)\times[0,T].
\end{align}
\end{subequations}
Our goal in this section is then to obtain new asymptotic models that approximate the main dynamics of the previous free boundary problem up to $\mathcal{O}(\varepsilon^2)$. Using the diffeomorphism
$$
\Psi(x_1,x_2,t)=(x_1,x_2+\varepsilon h(x_1,t)),
$$
to fix the domain together with \eqref{eq:ansatz} and the same precedure as before, we find the following nonlinear and nonlocal wave equation that captures (up to an error of order $\mathcal{O}(\varepsilon^2)$) the dynamics of a gravity-capillary surface wave in a fluid with odd viscosity
\begin{align}\label{eq:final_eq}
h_{tt}&=-\Lambda h-\beta\Lambda^3h+\alpha \Lambda \partial_1 h_{t}+\varepsilon\left[-\Lambda\left((H h_t)^2\right)+\partial_1\left(\comm{H}{h}\Lambda h\right)\right]\nonumber\\
&\quad +\varepsilon\partial_1\left[-\alpha\comm{H}{h}\Lambda \partial_1 h_{t}+\beta\comm{H}{h}\Lambda^3 h\right]&&\text{ on }\Gamma\times[0,T].
\end{align}
Using the previous far field change of variables 
$$
\xi=\alpha-t,\quad \tau=\varepsilon t,
$$
and abusing notation, we have that $f=\Lambda h$ in the far field variables solves
\begin{multline}\label{eq:univiscous3}
\varepsilon f_{t}=\mathcal{M}\partial_1 f+ \mathcal{M}H f+(\alpha-\beta)\mathcal{M}H\partial_1^2f\\
+\varepsilon\mathcal{M}\left\{-2f\partial_{1}f-\Lambda\comm{H}{\Lambda^{-1}f}f+(\alpha-\beta)\Lambda\comm{H}{\Lambda^{-1}f}\Lambda^2 f\right\},
\end{multline}
where the operator $\mathcal{M}$ is defined as follows
$$
\mathcal{M}=\left(2+\alpha\Lambda\right)^{-1}.
$$

\section{Models of surface waves in perfect fluids: truncation in the steepness}\label{CGSW}
Let us emphasize that when $\alpha_2=\alpha_1=\alpha=0$, equations \eqref{models2} and \eqref{eq:final_eq} recover the quadratic $h-$model in \cite{aurther2019rigorous, matsuno1992nonlinear, matsuno1993two, matsuno1993nonlinear, AkMi2010, AkNi2010}
\begin{align}\label{eq:final_eq2}
h_{tt}&=-\Lambda h-\beta\Lambda^3h+\varepsilon\left[-\Lambda\left((H h_t)^2\right)+\partial_1\left(\comm{H}{h}\Lambda h\right)\right] +\varepsilon\beta\partial_1\comm{H}{h}\Lambda^3 h&&\text{ on }\Gamma\times[0,T].
\end{align}
As a consequence, we find that the corresponding unidirectional wave equation for the case of inviscid fluids is
\begin{equation}\label{eq:uniinviscid}
2\varepsilon f_{t}=\partial_1 f+ H f-\beta H \partial_1^2 f
-\varepsilon \bigg{\lbrace}2f\partial_1 f+\Lambda\comm{H}{\Lambda^{-1}f}f +\beta\Lambda\comm{H}{\Lambda^{-1}f}\Lambda^2 f\bigg{\rbrace}.
\end{equation}
Once that we have established the asymptotic models for gravity-capillary surface waves, in the next section we turn our attention to the case of internal waves.

\section{Models of internal waves in perfect fluids: truncation in the nonlocality}\label{GS}
In this section we are going to consider the case of internal waves in perfect fluids. This problem is more challenging than the case of surface waves and it is related to some classical hydrodynamical instabilities. Such hydrodynamical instabilities could potentially even start a finite time singularity formation in the flow.

The problem of a finite time singularity formation for a free boundary problem is challenging due to the fact that, in addition to the blow-up of the solution happening in the bulk of the fluid, now a singularity can possibly occur at the boundary of the domain. In particular, we have to study the geometrical properties of the free boundary to discard possible pathological or unstable behavior (cusp formation, self-intersection of the free boundary, roll-up) leading to blow-up scenarios.

Two of these instabilities occuring at interfaces between two different fluids are the Rayleigh-Taylor and the Kelvin-Helmholtz instabilities. The Rayleigh-Taylor (RT) instability (named after Rayleigh \cite{Ra1878} and Taylor \cite{Ta1950}) of an interface between two fluids of different densities occurs when the high-density fluid is placed over a lower-density fluid under gravity.

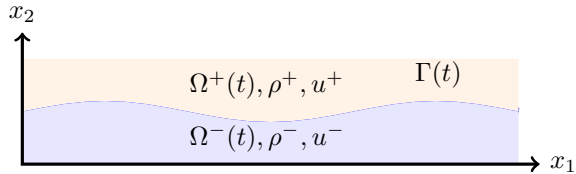
\begin{figure}[h]
%
\begin{center} 
\begin{tikzpicture}[domain=0:3*pi, scale=0.7] 
\draw[line width=1mm, smooth, color=blue] plot (\x,{0.2*sin(\x r)+1});
\fill[orange!10] plot[domain=0:3*pi] (\x,2) -- plot[domain=3*pi:0] (\x,{0.2*sin(\x r)+1});
\fill[blue!10] plot[domain=0:3*pi] (\x,0) -- plot[domain=3*pi:0] (\x,{0.2*sin(\x r)+1});
\draw[very thick,<->] (3*pi+0.4,0) node[right] {$x_1$} -- (0,0) -- (0,2.5) node[above] {$x_2$};
\node[right] at (2*pi+1,1.7) {$\Gamma(t)$};
\node[right] at (3,0.5) {$\Omega^-(t),\rho^-,u^-$};
\node[right] at (3,1.5) {$\Omega^+(t),\rho^+,u^+$};
\end{tikzpicture}   
\end{center}
\vspace{-.1 in}
\caption{\emph{A representation of an internal wave. The blue curve is an illustration of the interface $\Gamma(t)$, separating both fluids. The fluid on top of  $\Gamma(t)$ has density $\rho^+$ and velocity field $u^+$, while the fluid on the bottom has density $\rho^-$ and velocity $u^-$. }}
 \label{fig0}
\end{figure}

In the case where the two fluids present a good stratification, \emph{i.e.} the lighter fluid is on top, or even when they have the same density, the flow can still be unstable due to the so-called the Kelvin-Helmholtz (KH) instability. The KH instability (named after Lord Kelvin \cite{thomson1871xlvi} and von Helmholtz \cite{helmholtz1868uber}) can occur when there is a velocity difference across the interface between two fluids.

Due to the ubiquity of these physical phenomena, the problem of obtaining simple models that capture the main part of the dynamics of a two phase Euler flow is very important.

\subsection{The Birkhoff-Rott formulation of the Euler system}
The two-phase Euler equations modelling the motion of two incompressible, irrotational and inviscid fluids with labels $+$ and $-$ (see Figure 1) can be written as
\begin{align}
\label{euler}
\rho^\pm\left(u_t^\pm+(u^\pm\cdot\nabla)u^\pm\right)+\nabla p^\pm+g\rho^\pm e_2&=0\quad x\in \Omega^\pm(t),\,t\in[0,T]\\
\partial_t \rho^\pm+ u^\pm\cdot\nabla \rho^\pm&=0\quad x\in \Omega^\pm(t),\,t\in[0,T]\\
\nabla\cdot u^\pm&=0\quad x\in \Omega^\pm(t),\,t\in[0,T]\\
\nabla\times u^\pm&=0\quad x\in \Omega^\pm(t),\,t\in[0,T].
\end{align}
We denote the jump of $f(x,t)$ across $\Gamma(t)$ by
$$\jump{f}=f^+-f^-\,.$$
The previous system is supplemented with the following jump conditions:
\begin{equation}\label{jump_cond}
\jump{u\cdot n}=0\,, \ \ \ \jump{p}=\gamma\mathcal{K} \,, \ \ \text{ on } \Gamma(t) \,,
\end{equation}
where $n$ denotes the outward pointing unit normal to $\Omega^-$, $\gamma>0$ is the surface tension coefficient and $\mathcal{K}$ is the curvature. Here, the fluid with label $+$ (resp. $-$) has density $\rho^+$ and pressure $p^+$ (resp. $\rho^-$ and $p^-$) and occupies the volume $\Omega^+$ (resp. $\Omega^-$). We assume that $\bbR^2=\Omega^+\cup\Omega^-$ and the fluids are separated by the interface $\Gamma(t)$. This interface is parametrized as
$$
z(\alpha,t)=(z_1(\alpha,t),z_2(\alpha,t)).
$$
The curve $z$ is transported by the fluids. Furthermore, as the shape of the bulks of the fluids only depends on the normal component of the fluid velocity, we can add any tangential contribution to the velocity without affecting the problem. Then, the evolution equation for the parameterization of $\Gamma(t)$ is written as
\begin{equation}\label{eqb}
z_t( \alpha  ,t) =u(z( \alpha ,t),t)+c( \alpha ,t) \partial_\alpha z(\alpha ,t),
\end{equation}
where $c$ is an arbitrary function that we will fix below.

Due to the fact that each fluid is irrotational, the vorticity $\omega$ is supported in the interface between them and takes the form of the following measure
$$
\langle \omega ,\varphi\rangle = \int_ \mathbb{R}   \varpi(\beta,t)  \varphi( z( \beta , t)) d \beta \,,
$$
where $\varpi$ is a scalar function defined on $\Gamma(t)$. 

In what follows we are going to summarize how to obtain reduce the previous Euler system to a system of three nonlocal PDEs, one for the scalar function $\varpi$ and two PDEs for the parametrization $z$.

Due to the incompressibility of the flow we can use the Biot-Savart law to recover the velocity field from the vorticity
$$
u^\pm(x_1,x_2,t)=\frac{1}{2\pi}\text{P.V.}\int_\bbR \varpi(\beta) \left(-\frac{x_2-z_2(\beta,t)}{|(x_1,x_2)-z(\beta,t)|^2}, \frac{x_1-z_1(\beta,t)}{|(x_1,x_2)-z(\beta,t)|^2}\right)d\beta.
$$
Taking limits appropriately in the Biot-Savart integral, we obtain that
$$
\varpi=-\jump{u\cdot \partial_\alpha z}.
$$
Then, we can define the equation for the curve as
\begin{equation}\label{eqcurve}
z_t( \alpha  ,t) =\frac{1}{2\pi}\text{P.V.}\int_\bbR \varpi(\beta) \left(-\frac{z_2(\alpha,t)-z_2(\beta,t)}{|z(\alpha,t)-z(\beta,t)|^2}, \frac{z_1(\alpha,t)-z_1(\beta,t)}{|z(\alpha,t)-z(\beta,t)|^2}\right)d\beta+c( \alpha ,t) \partial_\alpha z(\alpha ,t).
\end{equation}
This is the so-called Birkhoff-Rott equation for a vortex sheet. Since the velocity is irrotational, there exist two velocity potentials $\phi^\pm$ satisfying
$$
u^\pm=\nabla \phi^\pm.
$$ 
and as a consequence
\begin{equation}\label{ss1}
\varpi = - \partial _ \alpha \jump{\phi} \,.
\end{equation} 
Using the ideas in \cite{cordoba2010interface}, we can further obtain that 
\begin{align}
\varpi_t&=-\partial_\alpha\bigg{[}- \frac{2\jump{p}}{\rho^++\rho^-} +
 \frac{A}{4\pi^2} \left|\int \varpi( \beta) \frac{(z( \alpha ,t) - z(\beta,t))^\perp}{| z( \alpha ,t) - z(\beta,t)|^2} d\beta\right|^2
 -\frac{A}{4} \frac{\varpi( \alpha ,t)^2}{|{\partial_\alpha} z( \alpha ,t)|^2}
\nonumber\\
&\quad \qquad
+  \frac{A}{\pi}  \int \varpi( \beta) \frac{(z( \alpha ,t) - z(\beta,t))^\perp}{| z( \alpha ,t) - z(\beta,t)|^2}  \cdot \partial_ \alpha z(\alpha ,t) c( \alpha ,t) d\beta 
\nonumber \\
&\qquad \qquad
-c( \alpha ,t) \varpi( \alpha ,t)  -2 A g z_2 \bigg{]} \nonumber \\
& \qquad \qquad
+\frac{A}{\pi} \partial_t\bigg{[}
 \int \varpi( \beta) \frac{(z( \alpha ,t) - z(\beta,t))^\perp}{| z( \alpha ,t) - z(\beta,t)|^2}  \cdot \partial_ \alpha z(\alpha ,t)  d\beta  
\bigg{]} \,,
\label{eqvorticity}
\end{align}
where $A$ denotes the Atwood number
$$
A=\frac{\rho^+-\rho^-}{\rho^++\rho^-}.
$$
As a consequence, the two-phase irrotational and incompressible Euler system in its Birkhoff-Rott kernel formulation is reduced to the previous two equations for $z$ \eqref{eqcurve} and the scalar equation for $\varpi$ \eqref{eqvorticity}.

In the rest of this section we are going to simplify the equations \eqref{eqcurve} and \eqref{eqvorticity} while retaining the main dynamics.

Also, let us remark that, at least formally, the case with vacuum on top,
$$
A=-1,
$$
corresponds with the surface water waves.
\subsection{Models for the case of a curve: the $z-$model}
Following \cite{GS}, we first take $c\equiv 0$ in \eqref{eqcurve}. Furthermore, approximating
$$
\frac{z(\alpha,t)-z(\beta,t)}{\alpha-\beta}\approx \partial_\alpha z(\alpha,t)
$$
and invoking Tricomi's identity 
\begin{equation}\label{tricomi}
2H(fHf)=(Hf)^2-f^2 \,,
\end{equation} 
we find the so-called $z-$model\footnote{Note the convention for $u^\perp$ that we are using in this paper is different than the one in \cite{GS}.}
\begin{subequations}\label{z-model}
\begin{align}
z_t(  \alpha ,t) & = -\frac{1}{2}H\varpi(\alpha,t) \frac{(\partial_\alpha z(\alpha,t))^\perp}{|\partial_\alpha z(\alpha,t)|^2}  \,, \label{z-model1} \\
\varpi_t(  \alpha ,t) &= -\partial_\alpha\bigg{[}\frac{A}{2} \frac{1}{|\partial_\alpha z( \alpha ,t)|^2}H\left(\varpi(\alpha,t)H\varpi( \alpha ,t)\right)
 - \frac{2\jump{p}}{\rho^++\rho^-}  -2 A g z_2 \bigg{]}\,.  \label{z-model2}
\end{align}
\end{subequations}
As explained in \cite{GS}, the previous approximation for $\partial_\alpha z$ can be understood as a localization in the singular integral operator to points near $\alpha.$ 

This system (see also \cite{kuznetsov1993surface}) has the advantage of being easily computable while at the same time being able to describe rather intrincate Rayleigh-Taylor instability dynamics. In fact, Canfield, Denissen, Francois, Gore, Rauenzahn, Reisner \and Shkoller \cite{canfield2020comparison} have compared with real experiments and it has been found that the $z-$model performed better than previous models for all cases that were explored. As a consequence, the $z-$model capture the main physical properties of several real experiments as well as some theoretical predictions available in the physics literature \cite{GS}. Moreover, the computational costs of simulating our models are negligible compared to the costs of simulating the whole two-phase Euler equations in the RT unstable case \cite{GS,ramani2020multiscale}.

The previous system can be written as a system of nonlinear wave equations. Indeed, using that
$$
H\varpi=-2z_t\cdot (\partial_\alpha z)^\perp\text{ and }\varpi=2H(z_t\cdot (\partial_\alpha z)^\perp),
$$
and taking a time derivative, we find the equivalent system
\begin{align}
z_{tt}&= -\Lambda\bigg{[}\frac{A}{|\partial_\alpha z|^2}H\left(z_t\cdot (\partial_\alpha z)^\perp H(z_t\cdot (\partial_\alpha z)^\perp)\right)
 + \frac{\jump{p}}{\rho^++\rho^-}  + A g z_2 \bigg{]} \frac{(\partial_\alpha z)^\perp}{|\partial_\alpha z|^2} \nonumber\\ 
 &\quad +z_t\cdot (\partial_\alpha z)^\perp\left(\frac{(\partial_\alpha z_t)^\perp}{|\partial_\alpha z|^2}-\frac{(\partial_\alpha z)^\perp 2(\partial_\alpha z\cdot \partial_\alpha z_t)}{|\partial_\alpha z|^4}\right)\,.\label{eqztt}
\end{align}

\subsection{Models for the Kelvin-Helmholtz instability} We observe that the previous $z-$model can also be used to model the case of equal densities $\rho^+=\rho^-$, \emph{i.e.} $A=0$. This case formally corresponds to the case of the Kelvin-Helmholtz instability. Using the equations \eqref{z-model} with $A=\gamma\equiv0$ we find the $z-$model for the Kelvin-Helmholtz instability: 
\begin{subequations}\label{KH-model}
\begin{align}
z_t(  \alpha ,t) & = -\frac{1}{2}H\varpi(\alpha,t) \frac{(\partial_\alpha z(\alpha,t))^\perp}{|\partial_\alpha z(\alpha,t)|^2}  \,, \label{KHmodel1} \\
\varpi_t(  \alpha ,t) &= 0\,.  \label{KHmodel2}
\end{align}
\end{subequations}
Integrating the equation for $\varpi$, we find that
\begin{align}
z_t(  \alpha ,t) & = -\frac{1}{2}H\varpi_0(\alpha) \frac{(\partial_\alpha z(\alpha,t))^\perp}{|\partial_\alpha z(\alpha,t)|^2}  \,. \label{KHmodel21}
\end{align}

In what follows we are going to derive a refinement of \eqref{KHmodel21} for the case of the Kelvin-Helmholtz instability. To the best of our knowledge this model is new. After changing variables the Euler formulation in its Birkhoff-Rott formulation reads as follows
\begin{align}
z_t(  \alpha ,t) &= {\frac{1}{2\pi}} {P.V.} \int_\bbR \varpi(\alpha- \beta) BR(\alpha,\beta)d\beta + c( \alpha , t) \partial_ \alpha 
z (\alpha , t) \,,\label{eqwexact}\\
\varpi_t(\alpha,t)&=-\partial_\alpha\bigg{[}
 \frac{A}{4\pi^2} \left|\int_\bbR \varpi(\alpha- \beta)BR(\alpha,\beta)d\beta\right|^2
 -\frac{A}{4} \frac{\varpi( \alpha ,t)^2}{|{\partial_\alpha} z( \alpha ,t)|^2}
\nonumber\\
&\quad \qquad
+  \frac{A}{\pi}  \int_\bbR \varpi(\alpha- \beta) BR(\alpha,\beta)  \cdot \partial_ \alpha z(\alpha ,t) c( \alpha ,t) d\beta 
-c( \alpha ,t) \varpi( \alpha ,t)  -2 A g z_2 \bigg{]} \nonumber \\
& \qquad \qquad
+\frac{A}{\pi} \partial_t\bigg{[}
 \int_\bbR \varpi(\alpha- \beta) BR(\alpha,\beta)  \cdot \partial_ \alpha z(\alpha ,t)  d\beta  
\bigg{]} \,,
\label{eqwrexact}
\end{align}
where $c(\alpha,t)$ is an arbitrary reparametrization term and the Birkhoff-Rott kernel is given by
$$
BR(\alpha,\beta)=-\frac{(z(\alpha)-z(\alpha-\beta))^\perp}{|z(\alpha)-z(\alpha-\beta)|^2}.
$$
Thus, another way of looking at the previous computations is as follows: we observe that we can expand $BR$ as a Laurent series around $\beta=0$ and find that
$$
BR(\alpha,\beta)=O_{-1}(\alpha)\beta^{-1}+O_0(\alpha)+O_{1}(\alpha)\beta+....=\sum_{j\geq{-1}}O_j(\alpha)\beta^j.
$$
The computation in the previous section implies that the coefficient of the principal part of the Laurent series of $BR$ is given by
\begin{align}
O_{-1}(\alpha)=-\frac{\partial_\alpha ^\perp z(\alpha)}{|\partial_\alpha z(\alpha)|^2}.
\label{O-1}
\end{align}
Thus, our original $z-$model consists in replacing the Birkhoff-Rott kernel $BR$ by the principal part of its Laurent series. Then, in order to refine the  previous $z-$model, we want to compute further terms of the Laurent series of the function $BR$ around $\beta=0$. It's a long but straightforward computation to find that
\begin{align}
O_0(\alpha)&=-\frac{\partial_\alpha^\perp z(\alpha)}{|\partial_\alpha z(\alpha)|^4}\partial_\alpha z(\alpha)\cdot \partial_\alpha^2 z(\alpha)
-\frac{1}{2}\frac{(\partial^2 z(\alpha))^\perp}{|\partial_\alpha z(\alpha)|^2},\label{O0}
\end{align}
and
\begin{align}
O_1(\alpha)&=-\frac{1}{6}\frac{(\partial_\alpha^3 z(\alpha))^\perp}{|\partial_\alpha z(\alpha)|^2}+\frac{1}{3}\frac{\partial_\alpha^\perp z(\alpha)}{|\partial_\alpha z(\alpha)|^4}\partial_\alpha z(\alpha)\cdot\partial_\alpha^3 z(\alpha)+\frac{1}{2}\frac{\partial_\alpha^\perp \partial_\alpha z(\alpha)}{|\partial_\alpha z(\alpha)|^4}\partial_\alpha z(\alpha)\cdot\partial_\alpha^2 z(\alpha)\nonumber\\
&\quad+\frac{1}{4}\frac{\partial_\alpha^\perp z(\alpha)}{|\partial_\alpha z(\alpha)|^4}|\partial_\alpha^2 z(\alpha)|^2-\frac{\partial_\alpha^\perp  z(\alpha)}{|\partial_\alpha z(\alpha)|^6}(\partial_\alpha z(\alpha)\cdot\partial_\alpha^2 z(\alpha))^2.\label{O1}
\end{align}
Then, the refined $z-$model for the Kelvin-Helmholtz instability reads
\begin{align}
z_t(  \alpha ,t) &=-\frac{1}{2} H\varpi_0(\alpha) \frac{\partial_\alpha ^\perp z(\alpha,t)}{|\partial_\alpha z(\alpha,t)|^2}\nonumber\\
&\quad+\bigg{(}\frac{1}{6}\frac{(\partial_\alpha^3 z(\alpha,t))^\perp}{|\partial_\alpha z(\alpha,t)|^2}-\frac{1}{3}\frac{\partial_\alpha^\perp z(\alpha,t)}{|\partial_\alpha z(\alpha,t)|^4}\partial_\alpha z(\alpha,t)\cdot\partial_\alpha^3 z(\alpha,t)-\frac{1}{2}\frac{\partial_\alpha^\perp \partial_\alpha z(\alpha,t)}{|\partial_\alpha z(\alpha,t)|^4}\partial_\alpha z(\alpha,t)\cdot\partial_\alpha^2 z(\alpha,t)\nonumber\\
&\quad-\frac{1}{4}\frac{\partial_\alpha^\perp z(\alpha,t)}{|\partial_\alpha z(\alpha,t)|^4}|\partial_\alpha^2 z(\alpha,t)|^2+\frac{\partial_\alpha^\perp  z(\alpha,t)}{|\partial_\alpha z(\alpha,t)|^6}(\partial_\alpha z(\alpha,t)\cdot\partial_\alpha^2 z(\alpha,t))^2\bigg{)}\frac{1}{2\pi}\int_{\mathbb{R}}\varpi_0(\beta)\beta d\beta \,,\label{refzmodelKH}
\end{align}
where we have used that
$$
\int_\bbR \varpi_0(\alpha)d\alpha=0.
$$
This model, that should not be confused with the high order $z-$model described in \cite{ramani2020multiscale}, is new and, to the best of author's knowledge, has never been studied mathematically before.

\subsection{Models for the case of a graph: the $h-$model}
Due to the reparametrization invariance, we are allowed to chose the tangential component of the velocity $c$ in \eqref{eqcurve}. If our initial curve can be parametrized as a graph
$$
( \alpha , h(\alpha ,0))
$$
and we want to maintain this feature, we chose
\begin{equation}\label{c}
c( \alpha , t) = {\frac{1}{2\pi}} {P.V.} \int_\bbR \varpi(\alpha- \beta,t) \frac{z_2(\alpha,t)-z_2(\alpha-\beta,t)}{|z(\alpha,t)-z(\alpha-\beta,t)|^2}d\beta \,.
\end{equation} 
Then we have that
$$
(z_1)_t(  \alpha ,t) = {\frac{1}{2\pi}} {P.V.} \int_\bbR \varpi(\alpha- \beta,t) \frac{z_2(\alpha,t)-z_2(\alpha-\beta,t)}{|z(\alpha,t)-z(\alpha-\beta,t)|^2}d\beta(-1+ \partial_ \alpha z_1 (\alpha , t))=0 \,,
$$
and we conclude that
$$
z_1(\alpha,t)=z_1(\alpha,0)=\alpha,
$$
and the graph parametrization propagates. We make use of the following power series expansion for $ | \zeta | < 1$:
$$
\frac{1}{1 + \zeta ^2} = 1 - \zeta ^2 +\cdot\cdot\cdot  \,,
$$
together with the approximation
$$
\frac{h(\alpha)-h(\beta)}{\alpha-\beta}\approx \partial_\alpha h(\alpha).
$$
Then, considering only quadratic nonlinearities we find that
$$
BR(\alpha,\beta)=\left(\frac{-(h(\alpha)-h(\alpha-\beta))}{\beta^2+(h(\alpha)-h(\alpha-\beta))^2},\frac{\beta}{\beta^2+(h(\alpha)-h(\alpha-\beta))^2}\right)\approx \frac{1}{\beta}\left(-\partial_\alpha h(\alpha),1-(\partial_\alpha h(\alpha))^2\right).
$$
If we neglect terms of cubic order, we obtain the so-called $h-$model system of equations (see \cite{GS})
\begin{subequations}\label{hmodel}
\begin{align} 
h_t( \alpha ,t) & =  {\frac{1}{2}} H \varpi( \alpha ,t) \,, \\
\varpi_t(\alpha,t)&=-\partial_\alpha\bigg{[}
 \frac{A}{4} \left|H \varpi( \alpha ,t)\right|^2
 -\frac{A}{4} \varpi( \alpha ,t)^2 - \frac{2\jump{p}}{\rho^++\rho^-}   -2 A g h( \alpha ,t) \bigg{]} \,.
\end{align} 
\end{subequations}
Taking another time derivative, changing to dimensionless variables and using \eqref{tricomi}, we arrive at the so-called $h-$model derived in \cite{GS}
\begin{equation} 
h_{tt}  = A \Lambda h  -  \beta \Lambda^3 h
 - A \p_\alpha(H h_t h_t) \,.\label{wavenonlinear}
\end{equation} 
As before, this model is a bidirectional wave equation of nonlocal type. We can also obtain a unidirectional model using a standard far-field change of variables
$$
\xi=\alpha-t,\quad \tau=\varepsilon t,\quad \zeta(\xi,\tau)=\varepsilon h(\xi,\tau)
$$
where $\varepsilon$ is small parameter. Then, going back to our previous notation for the space and time variables, we find that, up to an $O(\varepsilon^2)$ correction, $f=\partial_1\zeta$ solves
\begin{equation} 
-2\varepsilon f_{t}+\partial_1 f = A H f +  \beta H\partial_1^2 f
 - A\varepsilon \p_1(H f f) \,.\label{wavenonlinearuni}
\end{equation} 
Let us emphasize that the case $A=-1$ corresponds to surface water waves. Remarkably, a very similar equation also appears in \cite{abanov2018odd} when studying the motion of surface waves in a non-newtonian fluid with odd viscosity (see \cite{granero2021motion} for more details on this problem). Equations with the same nonlinearity appear in other contexts such as one-dimensional models of the surface quasi-geostrophic equation \cite{bae2015global,castro2008global,chae2005finite}, dislocations \cite{biler2010nonlinear} or the evolution of the density of the roots of polynomials \cite{granero2020nonlocal}.

\section{Numerical simulations for the inviscid water wave models}\label{Num}
\subsection{The $z-$model for bubbles and drops}
One of the main advantages of the previous asymptotic models is that they can be easily simulated using a standard computer. Furthermore, as the $z-$model does not require that the free boundary is a graph, it can be used to quickly simulate rising bubbles and falling drops. 

In this section we are going to show different simulations of the $z-$model corresponding to bubbles and drops. To do that we consider a closed curve and use a standard Fourier collocation method with $N=2^{11}$ spatial nodes for the interval $[-\pi,\pi]$. To advance in time we use a standard explicit Runge-Kutta (4,5) integrator. This algorithm is particularly well-adapted to the problem under consideration due to the fact that the nonlocal operators involved in the PDE are multiplier operators in the Fourier variables.

We fix $g=9.8m/s^2$ and consider a initial circle of radius one
$$
z(\alpha,0)=(\cosh(\alpha),\sin(\alpha))
$$
with initial velocity identically zero. The corresponding simulations are shown in figures \ref{figbub} and \ref{figdro}.

\begin{figure}[h]
\begin{center}
\includegraphics[scale=0.3]{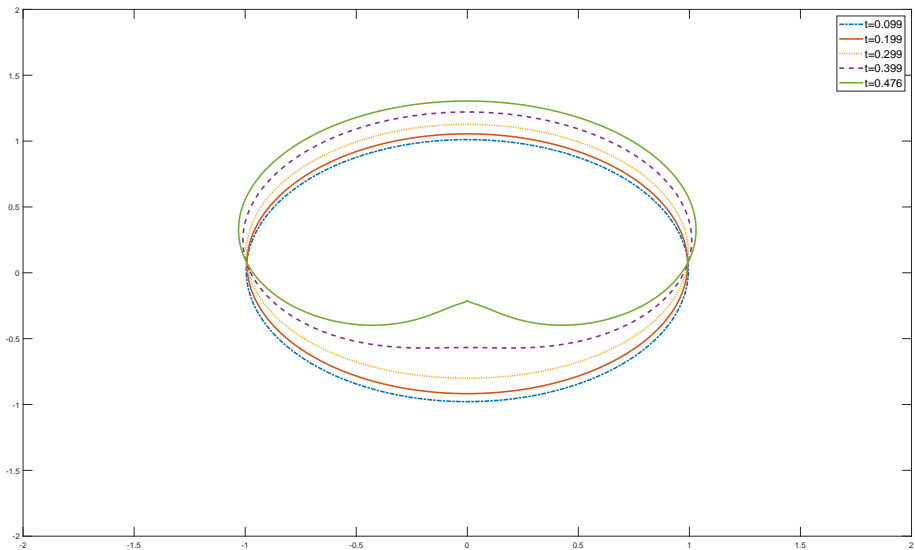}
\end{center}
\vspace{-.1 in}
\caption{Simulation of a rising bubble with $A=1/3$.}
\label{figbub}
\end{figure}

\begin{figure}[h]
\begin{center}
\includegraphics[scale=0.3]{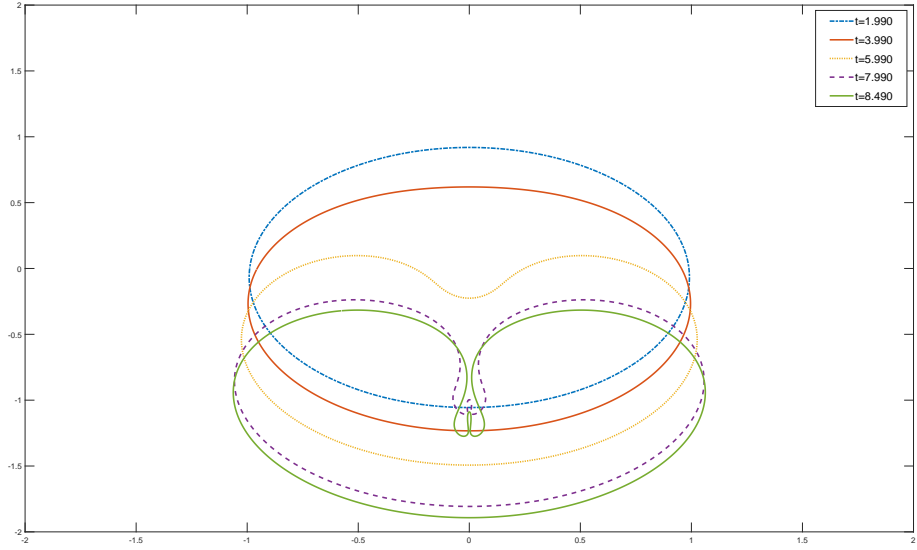}
\end{center}
\vspace{-.1 in}
\caption{Simulation of a falling drop with $A=-1/3$.}
\label{figdro}
\end{figure}

On one hand, we observe that the case of falling drops seems qualitatively more stable and seem to lead to self-intersection of the curve. On the other hand, the case of rising bubbles is more unstable and seem to lead to finite time singularities for the curvature. Remarkably, whether or not such finite time singularities actually occur for the system \eqref{eqztt} remains an open problem.

\subsection{Comparison between different models and the full water wave system}In this section we show several simulations corresponding to different models of inviscid flows and we compare the results with simulations of the full water wave system. We thank Professor Jon Wilkening for providing the simulations of the full water wave system used to create the figures in this section. These simulations of the full water wave system were obtained by Prof. Wilkening using the algorithm described in \cite{aurther2019rigorous}. 

We consider the following initial data
\begin{equation}\label{initialdata1}
h(x,0)= 1/3(\sin(x))\text{ and } h_t(x,0)=0
\end{equation}
and
\begin{equation}\label{initialdata2}
h(x,0)= 0.1\cdot(2/5\cdot(0.25\cdot(\sin(5x)+ 0.1\cdot\sin(8x))))\text{ and } h_t(x,0)=0
\end{equation}

\begin{figure}[h]
\begin{center}
\includegraphics[scale=0.3]{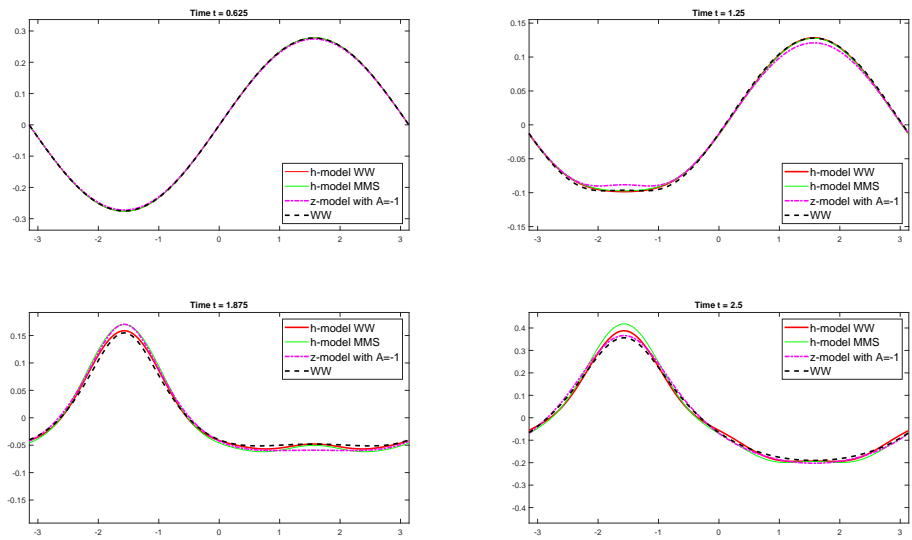}
\end{center}
\vspace{-.1 in}
\caption{Comparison between the different asymptotic models for the initial data in (\ref{initialdata1}).}
\label{figcom}
\end{figure}

\begin{figure}[h]
\begin{center}
\includegraphics[scale=0.3]{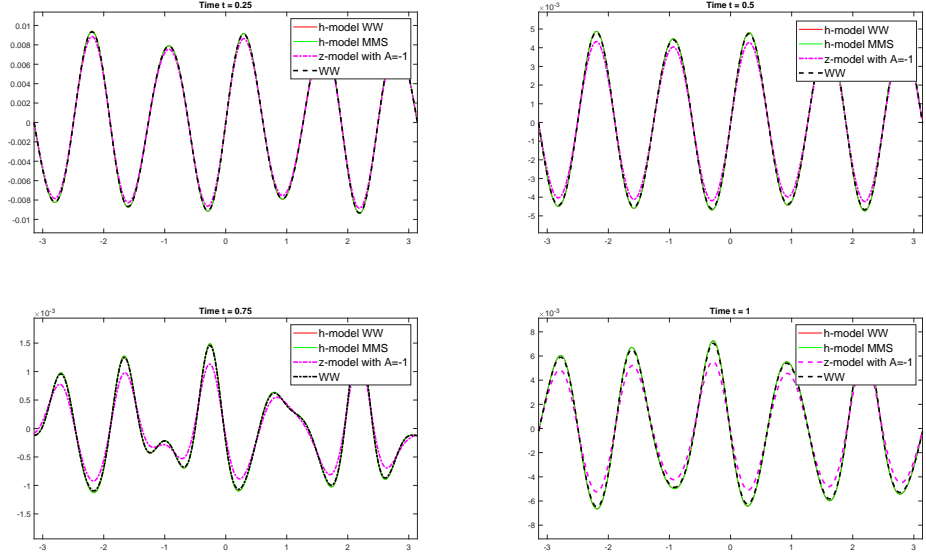}
\end{center}
\vspace{-.1 in}
\caption{Comparison between the different asymptotic models for the initial data in (\ref{initialdata2}).}
\label{figcom2}
\end{figure}

Figures \ref{figcom} and \ref{figcom2} show that the $z-$model is also qualitatively good when used to simulate the case of water waves even in the case where the free boundary remains a graph.

\section{The question of well-posedness}\label{new}
The question of well-posedness for the previous systems has been studied in several works. On the one hand, there are a number of works establishing the local well-posedness for the bidirectional models \eqref{eq:final_eq2}, \eqref{models2} and \eqref{eq:final_eq} (see \cite{aurther2019rigorous,granero2019well,granero2021motion,granero2021global}). Such a general well-posedness result holds for analytic initial data of arbitrary size and without any stability condition. On the other hand, there are also local existence results for the unidirectional models \eqref{eq:univiscous2} and \eqref{eq:univiscous3} (see \cite{granero2019well,granero2021motion,granero2021global}). These results holds for arbitrary initial data in the Sobolev class $H^s$ with $1<s$ large enough. In addition, for small initial data in $H^2$, it is known that the solution to \eqref{eq:univiscous2} is globally defined and decays towards equilibrium exponentially fast. Similarly, equation \eqref{wavenonlinear} was shown to be well-posed for initial data satisfying certain stability condition related to the stratification of the fluids \cite{GS}.

The purpose of this section is twofold: 
\begin{enumerate}
\item We will improve the global well-posedness result in \cite{granero2021global} for \eqref{eq:univiscous2}. Indeed, we will stablish the global well-posedness for small initial data in the Sobolev space $H^1$, thus, since the interface $h$ is one derivative smoother than $f$, our new result implies that $\|h\|_{H^2}$ is finite. At the same time our new result allows the surface wave to have unbounded curvature. To the best of the author's knowledge, this result is new.
\item We will prove the local existence of solution for the new model \eqref{wavenonlinearuni}. This result can be also applied to equation (51) in \cite{abanov2018odd}.
\end{enumerate}

\subsection{Well-posedness for the case of Newtonian fluids}
In this section we study the global well-posedness of the unidirectional model \eqref{eq:univiscous2} when the spatial domain is the torus $\mathbb{T}$ (or, equivalently, the interal $[-\pi,\pi]$ with periodic boundary conditions). More precisely, we have the following result

\begin{theorem}\label{theorem1} 
Let $\alpha_1,\alpha_2>0$ and $\beta\geq0$. Then given a zero mean $f_0\in H^{1}(\mathbb{T})$ satisfying
\begin{equation*}
\|f_0\|_{H^1(\mathbb{T})} \leq C(\alpha_1,\alpha_2), 
\end{equation*}
for a universal constant $C$, then the solution $f$ of equation \eqref{eq:univiscous2} is global
$$
f \in C([0,\infty),H^1(\mathbb{T}))\cap L^2([0,\infty);H^2(\mathbb{T})),
$$
and, furthermore, it decays towards the equilibrium.
\end{theorem}
\begin{proof}
The detailed proof of the global well-posedness result in \cite{granero2021global} is based on the modified energy 
$$
E_T(f)= \sup _{t\in[0, T]}\left\{e^{\alpha t}\|\hat{f}(t)\|_{\ell^1}+\|f(t)\|_{H^2(\mathbb{T})}\right\},
$$
where $\hat{f}$ denotes the Fourier series of $f$. Equipped with this definition of energy, the rest of the proof in \cite{granero2021global} is devoted to obtain an inequality of the form
$$
E_T(f)\leq C_0(f_0)+ P (E_T(f)),
$$
for certain polynomial $ P $ of degree larger than 1 and constant $C_0\pare{f_0}$ that depends on the initial data. Here we adopt a different approach. We are going to estimate directly the evolution of $\|f\|_{H^1(\mathbb{T})}$ and use the (nonlocal) parabolic character of the equation.

Without loss of generality, we take $\varepsilon=\alpha_1=\alpha_2=1$. As in \cite{granero2021global}, we start noticing that \eqref{eq:univiscous2} can be written equivalently as
\begin{multline}\label{NEW4v2}
2\varepsilon f_{t}=\mathcal{N}\partial_1f+2\mathcal{N}\partial_1^2f+ \mathcal{N}H f-\beta \mathcal{P}H \partial_1^2 f+\beta \mathcal{P}\Lambda \partial_1^2 f+\mathcal{P}\partial_1^3 f-\mathcal{P}\partial_1^4 f\\
- \mathcal{N}\bigg{\lbrace}2f\partial_1f+\Lambda\comm{H}{\Lambda^{-1}f}f +\beta\Lambda\comm{H}{\Lambda^{-1}f}\Lambda^2 f
\\
-\Lambda\comm{H}{f}\partial_1f+\partial_1\left(f\partial_1f \right)
+\Lambda\comm{\partial_1^2}{\Lambda^{-1}f} f
\bigg{\rbrace},
\end{multline}
where $\mathcal{P}$ denotes the operator 
$$
\mathcal{P}=(1-\partial_x^2)^{-1}.
$$
Now we perform the \emph{a priori} estimates in the $H^1$ Sobolev space.

Now we test \eqref{NEW4v2} against $-\partial_1^2 u$. Then we obtain that
\begin{align*}
\frac{d}{d t}\|f\|_{H^1}^2&=L+NL_1+NL_2+NL_3+NL_4+NL_5+NL_6,
\end{align*}
where
\begin{align*}
L&=\int_{\mathbb{T}}\lbrace\mathcal{N}\partial_1 f+2\mathcal{N}\partial_1^2 f+ \mathcal{N}H f-\beta \mathcal{P}H \partial_1^2 f\\
&\qquad+\beta \mathcal{P}\Lambda \partial_1^2 f+\mathcal{P}\partial_1^3 f-\mathcal{P}\partial_1^4 f\rbrace\Lambda^2 f dx,\\
NL_1&=-2\int_{\mathbb{T}}\mathcal{N}(f\partial_1 f)\Lambda^2 f dx,\\
NL_2&=-\int_{\mathbb{T}}\mathcal{N}\Lambda\comm{H}{\Lambda^{-1}f }f \Lambda^2 f dx,\\
NL_3&=-\beta\int_{\mathbb{T}}\mathcal{N}\Lambda\comm{H}{\Lambda^{-1}f}\Lambda^2 f\Lambda^2 f dx,\\
NL_4&=\int_{\mathbb{T}}\mathcal{N}\Lambda\comm{H}{f}\partial_1 f \Lambda^2 f dx,\\
NL_5&=-\int_{\mathbb{T}}\mathcal{N}\partial_1\left(f\partial_1 f \right)\Lambda^2 f dx,\\
NL_6&=-\int_{\mathbb{T}}\mathcal{N}\Lambda\comm{\partial_1^2}{\Lambda^{-1}f} f \Lambda^2 f dx.
\end{align*}
Integrating by parts, we find that
\begin{align*}
L&=-\|\mathcal{P}^{1/2}\partial_1^2f\|_{L^2}^2- \|\mathcal{P}^{1/2} \Lambda^{3/2}f\|_{L^2}^2-\beta \|\mathcal{P}^{1/2} \Lambda^{1/2}\partial_1^2f\|_{L^2}^2- \| \mathcal{P}^{1/2}\partial_3 f\|_{L^2}^2.
\end{align*}
We have to deal with the nonlinear terms. Integrating by parts, recalling the definition of $\mathcal{N}$ together with the Sobolev embedding, we can compute that
\begin{align*}
NL_1&=-\int_{\mathbb{T}}(1- \partial_1)\mathcal{P}^{1/2}(f^2)\mathcal{P}^{1/2}\partial_1^3 f dx,\\
&\leq C\|f\|_{H^1}^2\|\mathcal{P}^{1/2}\partial_1^3 f \|_{L^2}.
\end{align*}
We recall the commutator estimate in \cite{dawson2008decay})
\begin{align}\label{commutatorH}
\| \partial_1^\ell \comm{H}{U}\partial_1^m V\|_{L^p}\leq C\| \partial_1^{\ell+m}U\|_{L^\infty}\|V\|_{L^p}, && p\in(1,\infty), && \ell,m\in\mathbb{N}.
\end{align}
Then we have that
\begin{align*}
NL_2&=-\int_{\mathbb{T}}(1- \partial_1)\mathcal{P}^{1/2}H\comm{H}{\Lambda^{-1}f }f \mathcal{P}^{1/2}\partial_1^3 f dx,\\
&\leq C\|f\|_{H^1}^2\|\mathcal{P}^{1/2}\partial_1^3 f \|_{L^2}.
\end{align*}
Similarly,
\begin{align*}
NL_3&=\beta\int_{\mathbb{T}}(1- \partial_1)\mathcal{P}^{1/2}H\comm{H}{\Lambda^{-1}f}\Lambda^2 f\mathcal{P}^{1/2}\partial_1^3 f dx,\\
&\leq C\|f\|_{H^1}^2\|\mathcal{P}^{1/2}\partial_1^3 f \|_{L^2}.
\end{align*}
The term $NL_4$ can be handled in an analogous way, and we conclude
\begin{align*}
NL_4&\leq C\|f\|_{H^1}^2\|\mathcal{P}^{1/2}\partial_1^3 f \|_{L^2}.
\end{align*}
The term $NL_5$ can be estimated as follows
\begin{align*}
NL_5&=-\int_{\mathbb{T}}(1- \partial_1)\mathcal{P}^{1/2}\left(f\partial_1 f \right)\mathcal{P}^{1/2}\partial_1^3 f dx,\\
&\leq C\|f\|_{H^1}^2\|\mathcal{P}^{1/2}\partial_1^3 f \|_{L^2}.
\end{align*}
Finally, the remaining term $NL_6$ can be estimated as follows
\begin{align*}
NL_6&=-\int_{\mathbb{T}}(1- \partial_1)\mathcal{P}^{1/2}H\comm{\partial_1^2}{\Lambda^{-1}f} f \mathcal{P}^{1/2}\partial_1^3 f dx,\\
&\leq C\|f\|_{H^1}^2\|\mathcal{P}^{1/2}\partial_1^3 f \|_{L^2}.
\end{align*}
Collecting all these estimates, we find that
\begin{align*}
\frac{d}{d t}\|f\|_{H^1}^2&\leq C\|f\|_{H^1}^2\|\mathcal{P}^{1/2}\partial_1^3 f \|_{L^2}-\|\mathcal{P}^{1/2}\partial_1^3 f \|_{L^2}^2.
\end{align*}
Using that in the torus we have that
$$
\|f\|_{H^1}^2\leq C\|\mathcal{P}^{1/2}\partial_1^3 f \|_{L^2}^2,
$$
we conclude
\begin{align*}
\frac{d}{d t}\|f\|_{H^1}^2&\leq (C\|f\|_{H^1}-1)\|\mathcal{P}^{1/2}\partial_1^3 f \|_{L^2}^2.
\end{align*}
As a consequence, if $\|f_0\|_{H^1}$ is small enough, we have that 
\begin{align*}
\frac{d}{d t}\|f\|_{H^1}^2&\leq 0,
\end{align*}
and we conclude the global existence of solution using a standard continuation argument. Similarly, the uniqueness can be obtained using a contradiction argument together with the smallness of the initial data and the parabolic character of the problem. This concludes the result. 
\end{proof}

\subsection{Well-posedness for the case of inviscid fluids}
In this section we want to establish the local existence of solution for \eqref{wavenonlinearuni} for arbitrary analytic initial data. In order to do that we define the space
$$
\mathbb{A}_{\nu}=\left\{ f \in L^2, e^{\nu|n|}\hat{f}(n)\in L^1\right\}.
$$
This space is equipped with the norm
$$
\|u\|_{\mathbb{A}_{\nu}}=\|e^{\nu|n|}\hat{u}(n)\|_{L^1}.
$$
Then we have the following result:
\begin{theorem}\label{theorem2}
Let $f_0\in \mathbb{A}_{1}$ be the initial data. Then, defining
$$
\nu(t)=1-4\|f_0\|_{\mathbb{A}_{1}}t,
$$
there exists a local-in-time solution for equation \eqref{wavenonlinearuni}
$$
f\in L^\infty(0,T;\mathbb{A}_{\nu(t)}),
$$
for $0<T<(4\|f_0\|_{\mathbb{A}_{1}})^{-1}$.
\end{theorem}
\begin{proof}
Without loss of generality we take $\varepsilon=A=1$ in the proof. We focus on finding the appropriate estimates in the space $\mathbb{A}_{\nu(t)}$. We observe that these spaces ensure that the solution remain analytic and that the previous space is a Banach Algebra
\begin{align*}
\|fg\|_{\mathbb{A}_{\nu(t)}}&\leq \|f\|_{\mathbb{A}_{\nu(t)}}\|g\|_{\mathbb{A}_{\nu(t)}}.
\end{align*}
We compute
\begin{align*}
\frac{d}{dt}\left\|f\right\|_{\mathbb{A}_{\nu(t)}}&=\frac{d}{dt}\int_\mathbb{R} e^{\nu(t)|n|}|\widehat{f}(n,t)|dn\\
&=\int_{\mathbb{R}} \nu_t(t)|n|e^{\nu(t)|n|}|\widehat{f}(n,t)|dn+\int_{\mathbb{R}} e^{\nu(t)|n|}\text{Re}\left(\widehat{f}_t(n,t)\frac{\overline{\widehat{f}(n,t)}}{|\widehat{f}(n,t)|}\right)dn\\
&\leq \int_{\mathbb{R}}  \nu_t(t)|n|e^{\nu(t)|n|}|\widehat{f}(n,t)|dn+\left\|f_t\right\|_{\mathbb{A}_{\nu(t)}}.
\end{align*}
Then, if $0<\nu(t)$ is a decreasing function we find a regularizing contribution coming from $\nu_t$. This regularizing contribution is reflecting the fact that the strip of analyticity is shrinking.

As a consequence, if we take 
$$
0<\nu(t)=1-Ct,
$$
for $C>0$ to be fixed later, we find that

\begin{align*}
\frac{d}{dt}\left\|f\right\|_{\mathbb{A}_{\nu(t)}} &\leq -C\int_{\mathbb{R}}  |n|e^{\nu(t)|n|}|\widehat{f}(n,t)|dn+\left\|f_t\right\|_{\mathbb{A}_{\nu(t)}}\\
&\leq -C\|\Lambda f\|_{\mathbb{A}_{\nu(t)}}+\left\|\p_1(H f f)\right\|_{\mathbb{A}_{\nu(t)}}\\
&\leq -C\|\Lambda f\|_{\mathbb{A}_{\nu(t)}}+2\|\Lambda f\|_{\mathbb{A}_{\nu(t)}}\left\|f\right\|_{\mathbb{A}_{\nu(t)}}.
\end{align*}
Then, it is enough to take 
$$
C= 4\|f_0\|_{\mathbb{A}_{1}}
$$
to ensure that
\begin{align*}
\frac{d}{dt}\left\|f\right\|_{\mathbb{A}_{\nu(t)}}&\leq 0,
\end{align*}which ensures the uniform bound
$$
\left\|f\right\|_{\mathbb{A}_{\nu(t)}}\leq 0,
$$
for $T<(4\|f_0\|_{\mathbb{A}_{1}})^{-1}$. After this, the existence of solution follows from a standard mollifier approach together with Picard's theorem. Finally, the uniqueness follows from a standard continuation argument. This concludes the result.
\end{proof}

\section*{Acknowledgments} I thank Professor Jon Wilkening for providing his data for the simulations of the full water wave system. The author was supported by the project "Mathematical Analysis of Fluids and Applications" Grant PID2019-109348GA-I00 funded by MCIN/AEI/ 10.13039/501100011033 and acronym "MAFyA". This publication is part of the project PID2019-109348GA-I00 funded by MCIN/ AEI /10.13039/501100011033. This publication is also supported by a 2021 Leonardo Grant for Researchers and Cultural Creators, BBVA Foundation. The BBVA Foundation accepts no responsibility for the opinions, statements, and contents included in the project and/or the results thereof, which are entirely the responsibility of the authors.

\bibliographystyle{abbrv}

\end{document}